\documentclass{au}
\usepackage{amssymb,latexsym}
\usepackage{newlattice}

\DeclareMathOperator{\Princ}{Princ}

\newcommand{\Jp}[1]{\tup{J}^+(#1)}

\newcommand{\Pd}{P^{\,\tup{d}}}

\newtheorem{theorem}{Theorem}
\newtheorem{lemma}[theorem]{Lemma}
\newtheorem{corollary}[theorem]{Corollary}

\theoremstyle{definition} 
\newtheorem*{problem}{Problem}

\begin{document}
\title[On the set of principal congruences of a finite lattice]{Some preliminary results on\\ the set of principal congruences of a finite lattice}  
\author{G. Gr\"{a}tzer} 
\email[G. Gr\"atzer]{gratzer@me.com}
\urladdr[G. Gr\"atzer]{http://server.maths.umanitoba.ca/homepages/gratzer/}

\author[H. Lakser]{H. Lakser}
\email[H. Lakser]{hlakser@gmail.com} 
\address{Department of Mathematics\\University of Manitoba\\Winnipeg, MB R3T 2N2\\Canada}

\date{\today}
\dedicatory{To the memory of our friend of more than 50 years,
Bjarni J\'onsson}
\subjclass[2010]{Primary: 06B10.}
\keywords{congruence lattice, principal congruence,
join-irreducible congruence, 
principal congruence representable set.}

\begin{abstract}
In the second edition of the congruence lattice book, 
Problem 22.1 asks for a characterization of subsets
$Q$ of a finite distributive lattice $D$ such that
there is a finite lattice $L$ whose congruence lattice
is isomorphic to $D$ and under this isomorphism
$Q$ corresponds the the principal congruences of $L$.

In this note, we prove some preliminary results.
\end{abstract}

\maketitle

\section{Introduction}\label{S:Introduction}

For a finite lattice $K$, 
let $\J K$ and $\Princ K$ denote the (ordered) set of 
join-irreducible elements of $K$ and the 
principal congruences of $K$, respectively, 
and let 
\[
   \Jp K = \set{0,1} \uu \J K.
\]
Then for a finite lattice $L$,
\begin{equation}\label{E:cont}
   \Jp {\Con L} \ci \Princ L \ci \Con L,
\end{equation}
since every join-irreducible congruence is generated by a prime interval; furthermore, $\zero = \con{x,x}$ for any $x \in L$ and $\one = \con{0,1}$.

This paper continues G.~Gr\"atzer \cite{gG14}
(see also Section 10-6 of \cite{LTS1} and Part VI of~\cite{CFL2}),
whose main result is the following statement. 

\begin{theorem}\label{T:bounded}
Let $P$ be a bounded ordered set.
Then there is a bounded lattice~$K$ such that $P \iso \Princ K$.
If the ordered set $P$ is finite, 
then the lattice $K$ can be chosen to be finite.
\end{theorem}

The bibliography lists a number of papers related to this result. 
 
Let $D$ be a finite distributive lattice 
and let $ Q \ci D$.
We call $Q$ \emph{principal congruence representable}
(\emph{representable}, for short), if there is a finite lattice $L$
such that the following two conditions are satisfied:
\begin{enumeratei}
\item there is an isomorphism $\gy$ of $\Con L$ and $D$;
\item $\Princ L$ corresponds to $Q$ under $\gy$.
\end{enumeratei}

By \eqref{E:cont}, if $Q \ci D$ is representable, then 
\begin{equation}\label{E:repr}
   \Jp {D} \ci Q \ci D.
\end{equation}

We call $Q \ci D$ 
a \emph{candidate for principal representability}
(\emph{candidate} for short) if \eqref{E:repr} holds.

\begin{problem}\label{P:main}
Characterize representable sets for finite distributive lattices.
\end{problem}

\noindent This is Problem 22.1 in the book \cite{CFL2}.
In this paper, we prove some preliminary results. 
We hope that the techniques developed here 
will have further applications.

The first two results deal with the two extremal cases: 
$Q = D$ and $Q = \Jp {D}$.

\begin{theorem}\label{T:D}
Let $D$ be a finite distributive lattice.
Then $Q = D$ is representable.
\end{theorem}

\begin{theorem}\label{T:unit}
Let $D$ be a finite distributive lattice 
with a join-irreducible unit element.
Then $Q =  \Jp D$ is representable.
\end{theorem}

Let us call a finite distributive lattice $D$ 
\emph{fully representable}, if every candidate in $D$ is representable.

The third result finds the smallest 
distributive lattice that is not fully representable.

\begin{theorem}\label{T:8}
The smallest finite distributive lattice $D$
that is not fully representable has $8$ elements.
\end{theorem}

We use the notation as in \cite{CFL2}.
You can find the complete

\emph{Part I. A Brief Introduction to Lattices} and  
\emph{Glossary of Notation}

\noindent of \cite{CFL2} at 

\verb+tinyurl.com/lattices101+

For a bounded ordered set $Q$, 
the ordered set $Q^-$ is $Q$ with the bounds removed.

\subsection*{Outline}
Theorems~\ref{T:D} and~\ref{T:unit} are proved
in Section~\ref{S:Theorem1} and~\ref{S:Theorem2},
respectively. Section~\ref{S:Small} states 
(but does not prove) that 
all distributive lattices of size $\leq 7$ are fully representable; then it proves that the eight element Boolean lattice is not fully representable.

The version of this paper on arXive.org and
Research Gate contain the full Section~\ref{S:Small},
providing the diagrams for all distributive lattices
of size $4$--$7$, enumerating the candidates for representation, 
and representing each with a lattice.

There is an Addendum with some very recent results.

\section{Proving Theorem~\ref{T:D}}\label{S:Theorem1}

We need the following easy lemma from the folklore.

\begin{lemma}\label{L:folklore}
Let $L$ be a finite sectionally complemented lattice.
Then every congruence of $L$ is principal.
\end{lemma}

\begin{proof}
For a congruence $\bga$, let $a$ be the largest element in $L$
satisfying $\cng a = 0 (\bga)$. 
Then $\con{a,0} \leq \bga$. 
Conversely, the congruence $\cng u = v (\bga)$ in $L$ holds
if{} $\cng w = 0 (\bga)$, 
where $w$ is a sectional complement of $u \mm v$ in $u \jj v$.
Since $w \leq a$, it follows that $\con{a,0} = \bga$
and so $\bga$ is principal.
\end{proof}

This result is implicit in G. Gr\"atzer~\cite{GS59a}
(see also G. Gr\"atzer and E.\,T. Schmidt~\cite{GS62}),
since in a sectionally complemented lattice $L$
every congruence is a standard congruence and if $L$ is finite,
then every standard congruence is of the form $\con{0,s}$,
where $s$ is a standard element. 
See also \cite[Section III.2]{LTF}. 

We need now the main result of G.~Gr\"atzer and E.\,T. Schmidt 
\cite{GS62} (see also Theorem 8.5 in \cite{CFL2}).

\begin{theorem}\label{T:old}
Every finite distributive lattice $D$ 
can be represented as the congruence lattice 
of a finite sectionally complemented lattice $L$.
\end{theorem}

Theorem~\ref{T:old} and a reference to Lemma~\ref{L:folklore} 
completes the proof of Theorem~\ref{T:D}.

Rewriting \eqref{E:repr} for $D = \Con L$, we get
\begin{equation}\label{E:repr2}
   \Jp {\Con L} \ci Q \ci \Con L.
\end{equation}

\section{Proving Theorem \ref{T:unit}}\label{S:Theorem2}
We prove the following result.

\begin{theorem}\label{T:new2} \hfill
\begin{enumeratei}
\item Let $P$ be a finite ordered set with zero and unit. 
Then there is a finite lattice~$L$ such that
$P \iso \Princ L$. 

\item This finite lattice $L$ can be constructed such that
\begin{enumeratea}
\item the unit congruence, $\one$, is join-irreducible;
\item for every congruence $\bga > \zero$ of $L$,
if $\bga$ is principal, then it is join-irreducible.
\end{enumeratea}
\end{enumeratei}
\end{theorem}
\begin{proof}
The first statement is the finite case of Theorem~\ref{T:bounded}.
The second statement follows 
from the way the lattice $L$ is constructed; 
to~verify it, we have to briefly review this construction.

For a finite bounded ordered set $P$, 
we construct the lattice $L$ as follows. 

Let $0$ and $1$ denote the zero and unit of $P$, respectively.

We first construct the lattice $F$, the \emph{frame}, 
consisting of the elements \text{$o$, $i$} 
and the elements $a_p, b_p$ for every $p \in P$,
where $a_p \neq b_p$ for every $p \in P^-$
and $a_0 = b_0$, $a_1 = b_1$. These elements are ordered and
the lattice operations are formed as in Figure~\ref{F:F}. 
\begin{figure}[htb]
\centerline{\includegraphics[scale=1.0]{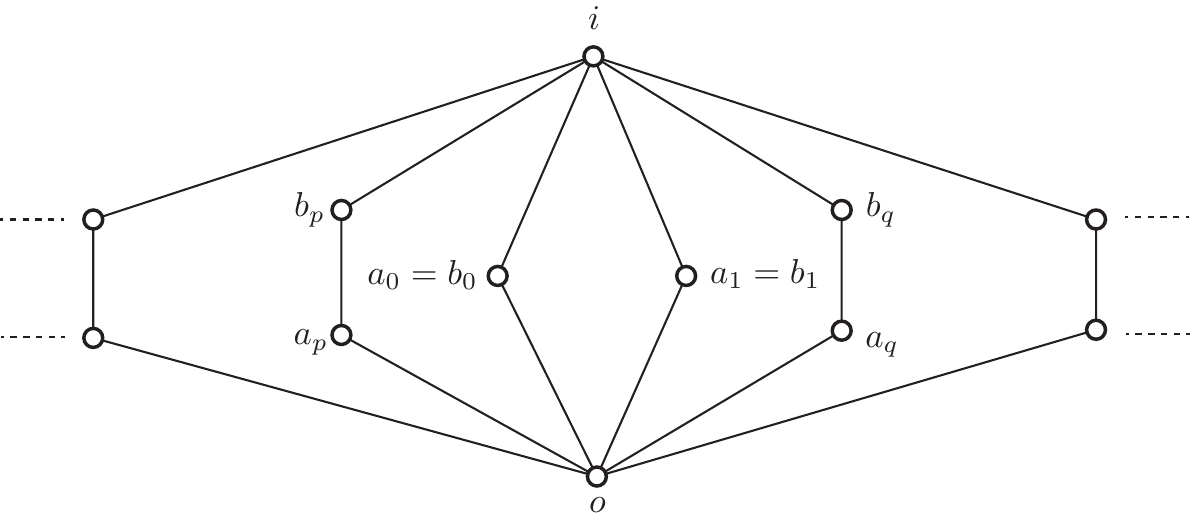}}
\caption{The frame lattice $F$}\label{F:F}
\end{figure}

The lattice $L$ is as an extension of $F$. 
We add five elements
to the sublattice 
\[
   \set{o, a_p, b_p, a_q, b_q, i}
\]
of $F$ for $p < q \in P^-$, 
as illustrated in Figure \ref{F:S},
to form the sublattice $S(p,q)$.  

\begin{figure}[b!]
\centerline{\includegraphics[scale=1.0]{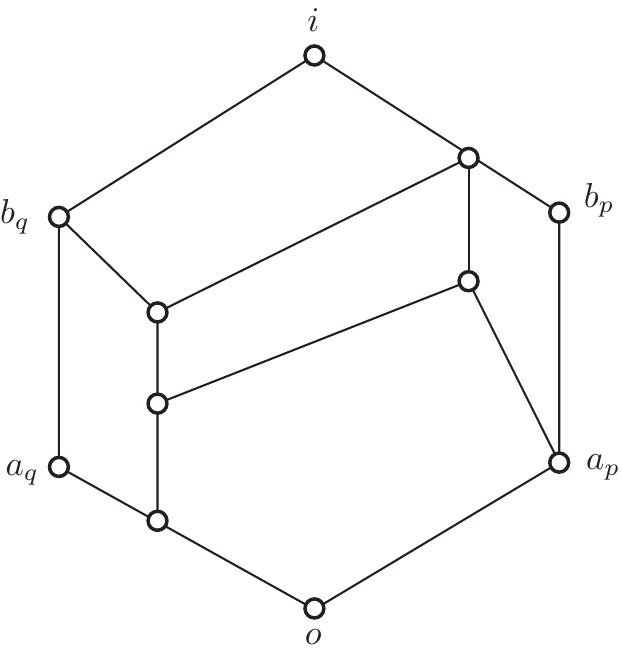}}
\caption{The lattice $S = S(p, q)$}\label{F:S}
\end{figure}

For $p \in \Pd$, let $C_p = \set{o < a_p < b_p < i}$
be a four-element chain.

Let $L$ be the lattice $F$ 
with all the $S(p,q)$ for $p < q \in P^-$ inserted.

There are two crucial observations about $L$.
\begin{enumeratei}
\item For any element $x\in L-\set{o, i, a_0, a_1}$, 
the sublattice $\set{o, i, a_0, a_1, x}$ is isomorphic to $\SM 3$.
\item For incomparable pairs $x,y \in L$ that belong
to an $S(p,q)$, it holds that $x,y$ are complementary.
\end{enumeratei}

It follows that the map
\begin{equation}\label{E:xx}
  p \mapsto 
      \begin{cases}
         \con{a_p, b_p}& \text{for $p \in P^-$};\\
         \zero & \text{for $p =0$};\\
         \one & \text{for $p =1$}
      \end{cases}
\end{equation}
is an isomorphism between $P$ and $\Princ L$. 
Observe that $a_p \prec b_p$ in $L$ and 
therefore $\con{a_p, b_p}$ is join-irreducible in $\Con L$.
Also $\zero = \con{o, o}$ and $\one = \con{o,i}$, 
veri\-fying statement (b).
\end{proof}

By Theorem~\ref{T:new2},
$Q = \Jp {\Con L} = \Jp {D} = \Princ L$
is representable.

There are many related constructions. We mention two.

\begin{enumeratei}
\item Let the finite lattice $K$ 
represent $Q \ci D = \Con L$.  
Form the lattice $L$ that
is an $\SM 3$ with $K$ replacing one of the atoms.  
Then $L$ represents $Q + \SC 1$
in $D + \SC 1$.

\item Similarly, with $Q, D, K$ as in (i), 
form $L'$ that is $\SC 2^2$ 
with $K$ replacing one of the atoms.  
Then $L'$ represents $Q \ci D$
with $\SC 2^2$ glued to the top.
\end{enumeratei}
 
\section{Small distributive lattices}\label{S:Small}

In this section, we will settle Problem~\ref{P:main} 
for distributive lattices of size $\leq 8$.

We start with a few easy statements.
 
Let $D$ be a finite distributive lattice 
and let $ Q \ci D$.
We call a candidate proper, if $Q \sci D$. 

The following statement is evident.

\begin{lemma}\label{L:easy}
Let $L$ be a finite lattice. If $\Jp {\Con L} = \Con L$,
then there is no proper candidate $Q$ for $\Con L$.
Therefore, $D$ is fully representable.
\end{lemma}

\begin{corollary}\label{C:chain}
Let $D$ be a finite chain. 
Then there is only one candidate $Q$ for $D$,
and it is representable.
\end{corollary}

\begin{lemma}\label{L:onejr}
Let $D$ be a finite distributive lattice. 
We assume that the unit element of $D$ is join-irreducible
and that $D$ has a unique join-reducible element $d \in D^-$.
Then $D$ has a unique proper candidate and it is representable.
\end{lemma}

\begin{proof}
$Q = D - \set{d}$ is the only proper candidate. 
Since $Q = \Jp{D}$, it is representable by Theorem~\ref{T:unit}.
\end{proof}

So let $D$ be a finite distributive lattice of size $\leq 8$.
If $D$ satisfies $\Jp{D} = D$, then the Problem 
is settled by Lemma~\ref{L:easy}, so we can assume that
\begin{equation}\label{E:repr3}
   \Jp {D} \sci D.
\end{equation}

\subsection{Distributive lattices of size $\boldsymbol{\leq}
\boldsymbol{7}$}
\label{S:7}
In this section, we state the first half of Theorem~\ref{T:8}.

\begin{theorem}\label{T:leq7}
Let $D$ be a distributive lattice of size $\leq 7$.
Then $D$ is fully representable.
\end{theorem}

We illustrate the proof with 
the distributive lattice $D = D_{7,5}$, see Figure~\ref{F:D75}.
\begin{figure}[htb]
\centerline{\includegraphics[scale = 1]{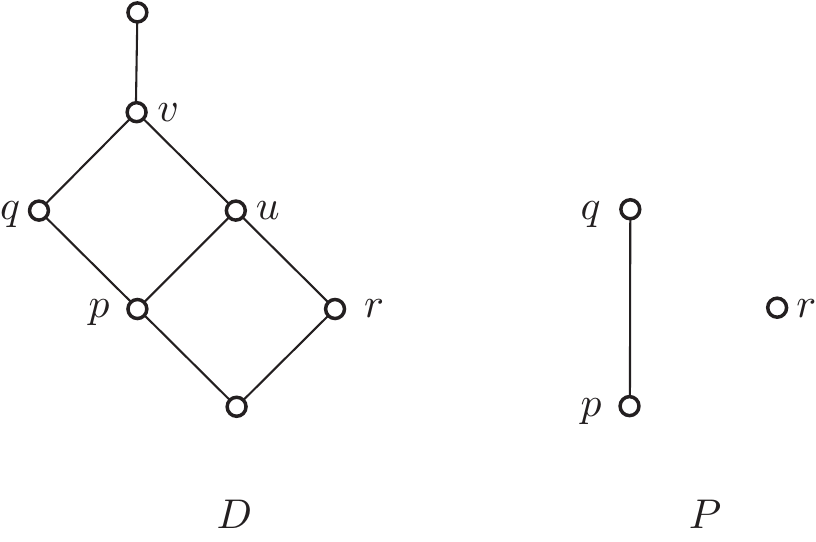}}
\caption{The distributive lattice $D = D_{7,5}$ 
and $P = \J D - \set{1}$}\label{F:D75}
\end{figure}
Let $u \prec v$ be the two join-reducible elements of $D$.
Then there are three proper candidates, $Q_{u,v} = D - \set{u,v},
Q_{u} = D - \set{u}$ and $Q_{v} = D - \set{v}$.
Since the unit element of $D$ is join-irreducible,
$Q_{u,v}$ is done by Theorem~\ref{T:unit}.

We represent the candidates $Q_{u}$ and $Q_{v}$ 
in Figures~\ref{F:D751} and~\ref{F:D75v}.

\begin{figure}[b!]
\centerline{\includegraphics{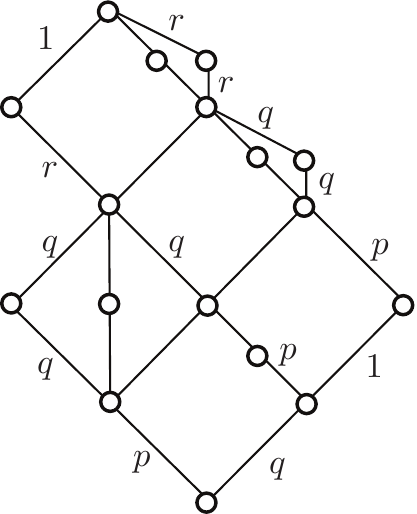}}
\caption{Representing $Q_{u}$ for $D = D_{7,5}$;
the congruence $q \jj r$ is principal, $p \jj r$ is not}\label{F:D751}
\end{figure}

\begin{figure}[htp]
\centerline{\includegraphics{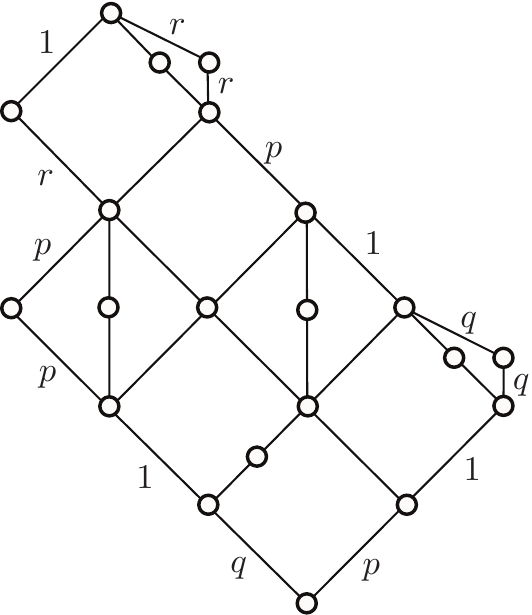}}
\caption{Representing $Q_{v}$ for $D = D_{7,5}$; the congruence
$p \jj r$ is principal, $q \jj r$ is not}\label{F:D75v}
\end{figure}

For the full proof Theorem~\ref{T:leq7}, 
see this paper in Research Gate:
\begin{quote}
\verb+http://tinyurl.com/principalcongruencesI+
\end{quote}
or in arXiv.org:
\begin{quote}
\verb+http://arxiv.org/abs/1705.05319+
\end{quote}

\subsection{A distributive lattice of size $8$}
\label{S:8}

Let $L$ be a finite lattice with congruence lattice 
$\Con L = \SB 3$, 
where $\SB 3$ is the eight-element Boolean lattice
with atoms $\bga_1, \bga_2, \bga_3$. 

By \eqref{E:repr},
\[
   \Jp {\SB 3} 
      = \set{\bga_1, \bga_2, \bga_3, \zero, \one} \ci \Princ L.
\]
\begin{lemma}\label{L:example}
There is no finite lattice $L$ with $\Con L = \SB 3$ and $\Jp {\SB 3} = \Princ L$.
\end{lemma}          

\begin{proof}
Let us assume that $L$ is a finite lattice 
with $\Con L = \SB 3$ and $\Jp {\SB 3} = \Princ L$.
Then in $L$, $\cng 0=1 (\one)$ and so $\cng 0=1 (\bga_1 \jj \bga_2 \jj \bga_3)$.
It follows that there is a finite chain $0=x_0 < x_1 < \dots < x_n = 1$
such that for every $0 \leq i < n$, 
there is $1 \leq j_i \leq 3$ satisfying $\cng x_i=x_{i+1} (\bga_{j_i})$
and $j_0 \neq j_1 \neq \dots \neq j_{n-1}$.
Without loss of generality, let $j_0 = 1$ and $j_1 = 2$.
Observe that 
\[
   \zero < \con{x_0, x_1} \leq \bga_{1}
\]
and so $\con{x_0, x_1} = \bga_{1}$ because $\bga_{1}$ is an atom.
Since $n > 1$, it follows that 
\[\bga_{1} \jj \bga_{2} = \con{x_0, x_2} \in \Princ L,\]
contrary to the assumption that $\Jp {\SB 3} = \Princ L$.
\end{proof}

Using the notation in the proof of Lemma~\ref{L:example}, we can prove more.

Since $\cng 0=1 (\bga_1 \jj \bga_2)$ fails, 
there must be a $2 \leq k < n$ so that $\cng x_k=x_{k+1} (\bga_{3})$. 
As~before, $\con{x_k, x_{k+1}} = \bga_{3}$. 
Then $j_{k-1} = 1$ or $j_{k-1} = 2$, say, \text{$j_{k-1} = 1$.} 
It follows that $\con{x_{k-1}, x_{k}} = \bga_{1}$ and 
\[\bga_{1} \jj \bga_{3} = \con{x_{k-1}, x_{k+1}} \in \Princ L.\]

So we conclude 
\begin{lemma}\label{L:two}
Let $L$ be a finite lattice $L$ with $\Con L = \SB 3$.
Then 
\[
   |\SB 3 - \Princ L| \leq 1.
\]

\end{lemma}

The converse is also true.

\begin{theorem}\label{T:B3}
Let $L$ be a finite lattice with $\Con L = \SB 3$.
Let $\Jp {\SB 3} \ci Q \ci \SB 3$.
Then $Q$ is representable if{}f $|\SB 3 - Q| \leq 1$.
\end{theorem}

\begin{proof}
If $Q$ is representable, then $|\SB 3 - Q| \leq 1$ by Lemma~\ref{L:two}.

Conversely, let $|\SB 3 - Q| \leq 1$. 
Then either $\SB 3 = Q$ or 
$\SB 3 = Q \uu \set{x}$, where x is a dual atom of $\SB 3$.
In the first case, we represent $Q$ by $L = \SB 3$. 
In the second case, we represent $Q$ by the chain $\SC 4$.
\end{proof}

This example is easy to modify to yield the following result.

\begin{corollary}\label{L:down}
Let $Q$ be representable. Then $Q^- = Q - \set{1}$ is not necessarily a down set.
\end{corollary}

Note the odd role planarity is playing in this paper.
All distributive lattices proven to be fully representable
are planar; the lattices we construct in this section are planar. The smallest distributive lattice that is not fully representable
is not planar.

\section*{Addendum, June 20, 2017}\label{S:Addendum}

After this paper was accepted for publication, 
we proved two theorems concerning the 
``odd role planarity is playing''
mentioned in the last paragraph of the previous section.
These results were independently obtained
by G\'abor Cz\'edli, and are presented here with his permission.

\begin{theorem}\label{T:planar}
Let $D$ be a finite distributive lattice.
If $D$ is fully representable, then it is planar.
\end{theorem}

\begin{proof}
Assume that $D$ is fully representable but not planar. 
Then $\J D$ contains a three-element antichain 
$\bga_0, \bga_1, \bga_2$. 
Let $\bga = \bga_0 \jj \bga_1 \jj \bga_2$.
Using that our antichain is in $\J D$, we obtain that 
$\bga \neq \bga_i \jj \bga_j$ for any $i, j \in\set{0,1,2}$. 
Since $D$ is fully representable, 
there is a finite lattice $L$ representing 
$Q = \J D \uu \set{\bga, \zero, \one}$.
In~particular, $\bga$ is principal in $L$, that is,
$\bga = \con{a,b}$. 
Let $X$ be a maximal chain 
$a = x_0 \prec \dots \prec x_n = b$ in $[a, b]$.

For each $i$ with $0 \leq i < n$, 
the interval $[x_i,x_{i+1}]$ is a prime interval in $L$;
thus $\con{x_i, x_{i+1}} \in \J D$.
Since $\con{x_i, x_{i+1}} \leq \bga$, there is then a
$j \in \set{0,1,2}$ such that
$\con{x_i,x_{i+1}} \leq \bga_j$.

Utilizing that $\JJm{\con{x_i,x_{i+1}}}{0 \leq i < n} = \bga$,  
for each $j \in \set{0,1,2}$, 
there exists an $i_j$ with $0 \leq i_j < n$ and 
$\bga_j \leq \con{x_{i_j},x_{i_j+1}} $.
Since $\set{\bga_0, \bga_1, \bga_2}$ 
is an antichain, we conclude that
$\con{x_{i_j},x_{i_j+1}} = \bga_j $.
We take the largest interval 
$[u, v] \ce [x_{i_0},x_{i_0+1}]$ of $X$ 
with $\con{u,v} = \bga_0$ and let $[u', v']$ 
be an interval of $X$ that contains $[u, v]$
and one more element of~$X$. 
Without loss of generality, 
we can assume that $u = x_k$, $v = x_l$, $k < l$,
$u' = u$, $v' = x_{l+1}$.
By the maximality of $[u, v]$, it follows that 
$\con{v,v'} \nleq \bga_0$.
But then $\con{v,v'} \leq \bga_1$ 
or $\con{v,v'} \leq \bga_2$,
say $\con{v,v'} \leq \bga_1$.
Thus the congruences $\con{u,v}=\bga_0$ 
and $\con{v,v'}$ are incomparable.
Consequently, $\con{u,v'}$ is join-reducible.
This is a contradiction,
since $\con{u, v'}$ is principal but 
$\con{u, v} < \con{u, v'} \leq \bga_0 \jj \bga_1 < \bga$,
and so $\con{u, v'}$ is not in $Q$.
\end{proof}

\begin{theorem}\label{T:9}
The planar distributive lattice $D = \SC 3^2$
is not fully representable.
\end{theorem}

\begin{proof}
Let $D = \SC 3^2$. Let $a < b$ and $c < d$ 
be the join irreducible elements of $D$, 
and let $Q = \set{0,a,b,c,d,e,1}$, where $e = a \jj c$.
We prove that $D$ is not fully representable 
because $Q \ci D$ is not representable.

Let us assume that a finite lattice $L$ represents $Q \ci D$.
Let $[x, y]$ be a maximal interval of $L$ 
with respect to the property $\con{x, y} = e$. 
Since $e < 1$, it follows that $[x, y] \neq L$. 
Without loss of generality, we can assume that $y$ 
is not the unit element of $L$. 
So there is an element $z \in L$ with $y \prec z$.
We cannot have $\con{y, z} \leq e$, 
because then $\con{x, z} = e$, 
in contradiction with the maximality of $[x, y]$.
So $\con{y, z} = b$ or $d$, say, $b$.
Then $\con{x, z} = b \jj e$, so $b \jj e$ is principal,
contradicting that $b \jj e \nin Q$.
\end{proof}

In the new manuscript \cite{gCc},
G\'abor Cz\'edli provides a characterization of
fully principal congruence representable distributive lattices
as planar distributive lattices in which only one dual atom
can be join-reducible.

\end{document}